\documentclass[a4paper,11pt]{article}
\usepackage{amsmath}
\usepackage{amsfonts}
\usepackage{amssymb}
\usepackage{amsthm}
\usepackage{amsbsy}
\usepackage{mathrsfs}
\usepackage{latexsym}
\usepackage{graphicx} 
\usepackage[latin1]{inputenc}			% Caracteres con acentos y smbolos latinos
\usepackage[english]{babel} 			% \usepackage[spanish]{babel}
\usepackage[T1]{fontenc}				% Para incorporar simbolos acentuados y otros smbolos
\usepackage{verbatim}					% Incluye texto en fuente mono-espaciada y comentarios
\usepackage{enumerate}					% Para el manejo de listas enumeradas y encajadas
\usepackage{array}						% Para darle formato al encabezado de las tablas
\usepackage{multicol}					% Allows to shift between two and one columns anywhere
\usepackage{paralist}					% Extended list environments 
\usepackage{multirow}
\usepackage{dsfont}
\usepackage{ctable}
\usepackage[all]{xy}					% Typesetting graphs and diagrams
\usepackage[dvips,colorlinks,citecolor = blue]{hyperref}
\usepackage{url}
\usepackage[numbers]{natbib}			% Citations and references \usepackage[round]{natbib}
\usepackage[hang,splitrule]{footmisc}

\DeclareMathOperator*{\argmax}{arg\,max}
\usepackage{algorithm}
\usepackage{algorithmic}

\theoremstyle{plain}

\newtheorem{theorem}{Theorem}

\newtheorem{lemma}{Lemma}

\theoremstyle{definition}

\hyphenrules{nohyphenation}
\usepackage{titling} % to reduce the white space that appears before the title
\usepackage[affil-it]{authblk}
\title{\textbf{Functional calibration estimation by the maximum entropy on the mean principle}}

\author[a,b]{Santiago Gallón\thanks{\texttt{santiagog@udea.edu.co}}}
\author[b]{Jean-Michel Loubes\thanks{\texttt{jean-michel.loubes@math.univ-toulouse.fr}}}
\author[b]{Fabrice Gamboa\thanks{\texttt{gamboa@math.univ-toulouse.fr}}}
\affil[a]{\small{Departamento de Matemáticas y Estadística, Facultad de Ciencias Económicas, Universidad de Antioquia, Medell\'in, Colombia.}}
\affil[b]{\small{Institut de Mathématiques de Toulouse, Universit\'e Toulouse III Paul Sabatier, Toulouse, France.}}

\begin{document}
\maketitle

\begin{abstract}
\noindent We extend the problem of obtaining an estimator for the finite population mean parameter incorporating complete auxiliary information through calibration estimation in survey sampling but considering a functional data framework. The functional calibration sampling weights of the estimator are obtained by matching the calibration estimation problem with the maximum entropy on the mean principle. In particular, the calibration estimation is viewed as an infinite dimensional linear inverse problem following the structure of the maximum entropy on the mean approach. We give a precise theoretical setting and estimate the functional calibration weights assuming, as prior measures, the centered Gaussian and compound Poisson random measures. Additionally, through a simple simulation study, we show that our functional calibration estimator improves its accuracy compared with the Horvitz-Thompson estimator.\vskip .1in
\noindent \textbf{Key words: } Auxiliary information; Functional calibration weights; Functional data; Infinite dimensional linear inverse problems; Survey sampling.
%\noindent \textbf{Subject Class. MSC-2010:}\mysubjclass
\end{abstract}

\section{Introduction}
\label{intro}

In survey sampling, the well-known calibration estimation method proposed by \citet{Deville-Sarndal-92} allows to construct an estimate for the finite population total or mean of a survey variable by incorporating complete auxiliary information on the study population in order to improve its efficiency. The main idea of the calibration method consists in modifying the standard sampling design weights $d_{i}$ of the unbiased Horvitz-Thompson estimator \citet{Horvitz-Thompson-52} by new weights $w_{i}$ close enough to $d_{i}$'s according to some distance function $\mathcal{D}(w,d)$, while satisfying a linear calibration equation in which the auxiliary information is taken into account. %The sources of this information may come, for example, from census data, administrative registers, and previous surveys \citep[e.g.,][]{Deville-Sarndal-92, Sarndal-92, Montanari-Ranalli-05}. 
The estimator based on these new calibration weights is asymptotically design unbiased and consistent with a variance smaller than the Horvitz-Thompson one.

The idea of calibration has been extended to estimate other finite population parameters, such as finite population variances, distribution functions and quantiles. See, for instance, \citet{Rao.et.al-90}, \citet{Kovacevic-97}, \citet{Theberge-99}, \citet{Singh-01}, \citet{Wu-Sitter-01}, \citet{Wu-03}, \citet{Harms-Duchesne-06}, \citet{Rueda.et.al-07}, \citet{Sarndal-07}, and references therein. Recent developments have also been conducted toward, for example, the approach of (parametric and non-parametric) non-linear relationships between the survey variable and the set of auxiliary variables for the underlying assisting model, and a broad classes of conceivable calibration constraints functions (\citet{Breidt-Opsomer-00, Wu-Sitter-01, Wu-03, Montanari-Ranalli-05}).

One interesting extension emerges when both the survey and auxiliary variables are considered as infinite dimensional objects such as random functions. This generalization relies on the fact that, due to improvements in data collection technologies, large and complex databases are being registered frequently at very fine time scales, regarded these as functional datasets. This kind of data are collected in many scientific fields as molecular biology, astronomy, marketing, finance, economics, among many other. A depth overview on functional data analysis can be found in \citet{Ramsay-Silverman-02}, \citet{Ramsay-Silverman-05} and \citet{Horvath-Kokoszka-12}. Functional versions of the Horvitz-Thompson estimator have been proposed recently by \citet{Cardot-Josserand-11} and \citet{Cardot-Degras-Josserand-11} for the cases of error free and noisy functional data, respectively.

The purpose of the present paper is to extend the problem of obtaining calibration sampling weights using functional data. This is conducted through the generalization of the work by \citet{Gamboa-Loubes-Rochet-11}, where the calibration estimation problem, which is considered as a
linear inverse problem following \citet{Theberge-99}, is matched with the maximum entropy on the mean approach under a finite dimensional setting. The maximum entropy on the mean principle applied to our goal focuses on reconstructing an unique posterior measure $\nu^{\ast}$ that maximizes the entropy $S(\nu\parallel\upsilon)$ between a feasible finite measure $\nu$ relative to a given prior measure $\upsilon$ subject to a linear constraint. %which in this paper we express as an infinite dimensional inverse problem, holds in mean.
Finally, the functional calibration sampling weights are defined as the mathematical expectation with respect to $\nu^{\ast}$ of a random variable with mean equal to the standard sampling design weights $d_{i}$. In this paper, we reconstruct $\nu^{\ast}$ adopting the random measure approach by \citet{Gzyl-Velasquez-11} under an infinite dimensional context.

The maximum entropy method on the mean  was introduced by \citet{Navaza-85, Navaza-86} to solve an inverse problem in crystallography, and has been further investigated, from a mathematical point of view, by \citet{Gamboa-89}, \citet{Dacunha-Gamboa-90} and \citet{Gamboa-Gassiat-97}. Complementary references on the approach are \citet{Mohammad-Djafari-96}, \citet{Marechal-99}, \citet{Gzyl-00}, \citet{Gzyl-Velasquez-11} and \citet{Golan-Gzyl-12}. Maximum entropy solutions, as an alternative to the Tikhonov's regularization of ill-conditioned inverse problems, provide a very simple and natural way to incorporate constraints on the support and the range of the solution \citet{Gamboa-Gassiat-97}, and its usefulness has been proven, e.g., in crystallography, seismic tomography and image reconstruction.

The paper is organized as follows. Sect.~\ref{sec:2}, presents the calibration estimation framework for the functional finite population mean. In Sect.~\ref{sec:3}, the connection between calibration and maximum entropy on the mean approaches is established, and the functional calibration sampling weights are obtained assuming two prior measures. In Sect.~\ref{sec:4}, the respective approximations of the functional maximum entropy on the mean estimators are derived. The performance of the estimator is studied through a simple simulation study in Sect.~\ref{sec:5}. Some concluding remarks are given in Sect.~\ref{sec:6}. Finally, the technical proofs of the technical results are gathered in the Appendix.

\section{Calibration estimation for the functional finite population mean}
\label{sec:2}

Let $U_{N}=\left\lbrace 1,\ldots,N\right\rbrace $ be a finite survey population from which a realized sample $a$ is drawn with fixed-size sampling design $p_{N}(a)=\mathbb{P}(A=a)$. Here $a\in\mathcal{A}$,
where $\mathcal{A}$ is the collection of all subsets $A$ of $U_{N}$ that contains all possible samples of $n_{N}$ different elements randomly drawn from $U_{N}$ according to a given sampling selection
scheme, and $\mathbb{P}$ a probability measure on $\mathcal{A}$. The first order inclusion probabilities, $\pi_{iN}=\mathbb{P}(i\in a)=\sum_{a\in A(i)}p_{N}(a)$, where $A(i)$ represents the set of samples that contain the $i$th element, are assumed to be strictly positive for all $i\in U_{N}$. See \citet{Sarndal-92} and \citet{Fuller-09} for details about survey sampling.

Associated with the $i$th element in $U_{N}$ there exists an unique functional random variable $Y_{i}(t)$ with values in the space of all continuous real-valued functions defined on $\left[0,T\right]$ with $T<+\infty$, $\mathcal{C}(\left[0,T\right])$. However, only the sample functional data, $Y_{i}(t)$, $i\in a$ are observed. Additionally, an auxiliary $q$-dimensional functional vector is available for each $i\in U_{N}$, $\boldsymbol{X}_{i}(t)=(X_{i1}(t),\ldots,X_{iq}(t))^{\top}\in \mathcal{C}\bigl(\left[0,T \right]^{q}\bigr)$ with $q\geq1$. The known functional finite population mean is denoted by $\boldsymbol{\mu}_{X}(t)=N^{-1}\sum_{i\in U_{N}}\boldsymbol{X}_{i}(t)$.

The main goal is to obtain a design consistent estimator for the unknown functional finite population mean, $\mu_{Y}(t)=N^{-1}\sum_{i\in U_{N}}Y_{i}(t)$, based on the calibration method. The idea consists in modify the basic sampling design weights, $d_{i}=\pi_{i}^{-1}$, of the unbiased functional Horvitz-Thompson estimator defined by $\hat{\mu}_{Y}^{HT}(t)=N^{-1}\sum_{i\in a}d_{i}Y_{i}(t)$, for new more efficient weights $w_{i}>0$ incorporating the auxiliary information. These weights must to be sufficiently close to $d_{i}$'s according to some dissimilarity distance function $\mathcal{D}_{a}(w,d)$ on $\mathbb{R}_{+}^{n}$, and satisfying the set of calibration constraints 
\[
N^{-1}\sum_{i\in a}w_{i}\boldsymbol{X}_{i}(t)=\boldsymbol{\mu}_{X}(t).
\]

The functional estimator for $\mu_{Y}(t)$ based on the calibration weights is expressed by the linear weighted estimator $\hat{\mu}_{Y}(t)=N^{-1}\sum_{i\in a}w_{i}Y_{i}(t)$. Different calibration estimators can be obtained depending on the chosen distance function \citet{Deville-Sarndal-92}. However, it is well known that, in the finite dimensional setting, all of calibration estimators are asymptotically equivalent to the one obtained through the use of the popular chi-square distance function $\mathcal{D}_{a}(w,d)=\sum_{i\in a}(w_{i}-d_{i})^{2}/2d_{i}q_{i}$,
%\[
%\mathcal{D}_{a}(w,d)=\sum_{i\in a}\dfrac{(w_{i}-d_{i})^{2}}{2d_{i}q_{i}},
%\]
where $q_{i}$ is an individual given positive weight uncorrelated with $d_{i}$.

Assuming a point-wise multiple linear regression model \citet{Ramsay-Silverman-05}, $Y_{i}(t)=\boldsymbol{X}_{i}(t)^{\top}\boldsymbol{\beta}(t)+\varepsilon_{i}(t)$, where $\varepsilon_{i}(t)$ is the $i$th zero-mean measurement functional error independent of $\boldsymbol{X}_{i}(t)$ with variance structure given by a diagonal matrix with elements $1/q_{i}$ unrelated to $d_{i}$, then the estimator for $\mu_{Y}(t)$ from the restricted minimization problem can be expressed as 
\[
\begin{split}
\hat{\mu}_{Y}(t)
& =\hat{\mu}_{Y}^{HT}(t)+\left\lbrace\boldsymbol{\mu}_{X}(t)-\hat{\boldsymbol{\mu}}_{X}^{HT}(t)\right\rbrace^{\top}\widehat{\boldsymbol{\beta}}(t),
\end{split}
\]
where $\hat{\boldsymbol{\mu}}_{X}^{HT}(t)=\sum_{i\in a}d_{i}\boldsymbol{X}_{i}(t)$ denotes the Horvitz-Thompson estimator for the functional vector $\boldsymbol{X}(t)$, and $\widehat{\boldsymbol{\beta}}(t)=\left\lbrace \sum_{i\in a}d_{i}q_{i}\boldsymbol{X}_{i}(t)\boldsymbol{X}_{i}(t)^{\top}\right\rbrace ^{-1}\sum_{i\in a}d_{i}q_{i}\boldsymbol{X}_{i}(t)Y_{i}(t)$
is the weighted estimator of the functional coefficient vector $\boldsymbol{\beta}(t)$, whose uniqueness relies on the existence of the inverse of the matrix $\sum_{i\in a}d_{i}q_{i}\boldsymbol{X}_{i}(t)\boldsymbol{X}_{i}(t)^{\top}$ for all $t$.

The calibration weights can be generalized allowing functional calibration weights ${w}_{i}(t)$ which can be obtained from the minimization of the generalized chi-square distance $\mathcal{D}_{a}^{\ast}(w,d)$, expressed below, subject to the functional calibration restriction
\begin{equation}\label{equ1}
N^{-1}\sum_{i\in a}w_{i}(t)\boldsymbol{X}_{i}(t)=\boldsymbol{\mu}_{X}(t).
\end{equation}

The existence of functional calibration weights is stated in the next theorem, which is a straightforward generalization of the finite dimensional results of \citet{Deville-Sarndal-92}.

\begin{theorem}\label{theorem.1}
Assume the existence of a functional vector $\boldsymbol{w}(t)=\left(w_{1}(t),\ldots,w_{n}(t)\right)^{\top}$ such that~\eqref{equ1} holds, and the inverse of the matrix $\sum_{i\in a}d_{i}q_{i}\left(t\right)\boldsymbol{X}_{i}(t)\boldsymbol{X}_{i}(t)^{\top}$. Then, for a fixed $t\in\left[0,T\right]$, $\hat{\boldsymbol{w}}(t)$ minimizes over $\mathcal{C}(\left[0,T\right]^{n})$ the generalized chi-square distance 
\[
\mathcal{D}_{a}^{\ast}(w,d)=\sum_{i\in a}\frac{\bigl(w_{i}(t)-d_{i}\bigr)^{2}}{2d_{i}q_{i}(t)}
\]
subject to~\eqref{equ1}, where the functional calibration weight $\hat{w}_{i}(t)$ for all $i\in a$ is given by
\[
\hat{w}_{i}(t)=d_{i}\left[1+q_{i}(t)\left\lbrace \boldsymbol{\mu}_{X}(t)-\hat{\boldsymbol{\mu}}_{X}^{HT}(t)\right\rbrace ^{\top}\left\lbrace \sum_{i\in a}d_{i}q_{i}(t)\boldsymbol{X}_{i}(t)\boldsymbol{X}_{i}(t)^{\top}\right\rbrace ^{-1}\boldsymbol{X}_{i}(t)\right].
\]
\end{theorem}

Note that, for this generalized setting, the functional calibration estimator for $\mu_{Y}(t)$ is expressed by
\[
\begin{split}
\hat{\mu}_{Y}(t)
& =N^{-1}\sum_{i\in a}\hat{w}_{i}(t)Y_{i}(t)=\hat{\mu}_{Y}^{HT}(t)+\left\lbrace \boldsymbol{\mu}_{X}(t)-\hat{\boldsymbol{\mu}}_{X}^{HT}(t)\right\rbrace ^{\top}\widehat{\boldsymbol{\beta}}(t),
\end{split}
\]
where
\[
\widehat{\boldsymbol{\beta}}(t)=\left\lbrace \sum_{i\in a}d_{i}q_{i}(t)\boldsymbol{X}_{i}(t)\boldsymbol{X}_{i}(t)^{\top}\right\rbrace ^{-1}\sum_{i\in a}d_{i}q_{i}(t)\boldsymbol{X}_{i}(t)Y_{i}(t),
\]
provided the inverse of the matrix $\sum_{i\in a}d_{i}q_{i}(t)\boldsymbol{X}_{i}(t)\boldsymbol{X}_{i}(t)^{\top}$ exists for all $t$.

\section{Maximum entropy on the mean for survey sampling}
\label{sec:3}

Let $(\widetilde{\mathcal{X}},\mathcal{F})$ be an arbitrary measurable space over which we want to search for an $\sigma$-finite positive measure $\mu$. The maximum entropy on the mean principle provides an efficient way of getting an estimator for some linear functional $\mu_{\widetilde{Y}}(t)=\int_{\mathcal{\widetilde{X}}}\widetilde{Y}(t)d\mu$ satisfying a known $q$-vector of functionals $\int_{\mathcal{\widetilde{X}}}\widetilde{\boldsymbol{X}}(t)d\mu=\boldsymbol{\mu}_{X}(t)$,
where $\widetilde{Y}(t)\colon\widetilde{\mathcal{X}}\to\mathcal{C}(\left[0,T\right])$ and $\widetilde{\boldsymbol{X}}(t)\colon\mathcal{\widetilde{X}}\to\mathcal{C}(\left[0,T\right]^{q})$.

A natural unbiased and consistent estimator of $\mu_{\widetilde{Y}}(t)$ is the empirical functional mean $\hat{\mu}_{\widetilde{Y}}(t)=\int_{\chi}\widetilde{Y}(t)d\mu_{n}=n^{-1}\sum_{i\in a}\widetilde{Y}_{i}(t)$, where $\mu_{n}=n^{-1}\sum_{i\in a}\delta_{T_{i}}$ is the corresponding empirical distribution with $T_{1},\ldots,T_{n}$ an observed random sample from $\mu$. Despite properties of this estimator, it may not have the smallest variance in this kind of framework. Therefore, incorporating prior functional auxiliary information the variance of an asymptotically unbiased functional estimator can be reduced applying the maximum entropy on the mean principle.  %\citep{Gamboa-Loubes-Rochet-11}

The philosophy of the principle consists in to enhance $\hat{\mu}_{\widetilde{Y}}(t)$ considering the maximum entropy on the mean functional estimator

\[
\hat{\mu}_{\widetilde{Y}}^{MEM}(t)=\int_{\chi}\widetilde{Y}(t)d\hat{\mu}_{n}^{MEM}=n^{-1}\sum_{i\in a}\hat{p}_{i}(t)\widetilde{Y}_{i}(t),\quad \text{for all } t\in\left[0,T\right],
\]
where $\hat{\mu}_{n}^{MEM}=n^{-1}\sum_{i\in a}\hat{p}_{i}(t)\delta_{T_{i}}$ is a weighted version of the empirical distribution $\mu_{n}$, with $\hat{\boldsymbol{p}}(t)=(\hat{p}_{1}(t),\ldots,\hat{p}_{n}(t))^{\top}$ given by the expectation of the independent $n$-dimensional stochastic process $\boldsymbol{P}(t)=(P_{1}(t),\ldots,P_{n}(t))^{\top}$ drawn from a posterior finite distribution $\nu^{\ast}$, $\hat{\boldsymbol{p}}(t)=\mathbb{E}_{\nu^{\ast}}\left[\boldsymbol{P}(t)\right]$ for all $t\in\left[0,T\right]$, where $\nu^{\ast}$ must to be close to a prior distribution $\upsilon$, which transmits the information that $\hat{\mu}_{n}^{MEM}$ must to be sufficiently close to $\mu_{n}$.

Therefore, the maximum entropy on the mean principle focuses on reconstructing the posterior measure $\nu^{\ast}$ that maximizes the entropy, over the convex set of all probability measures, $S(\nu\parallel\upsilon)=-D(\nu\parallel\upsilon)$ subject to the linear functional constraint holds in mean,
\[
\mathbb{E}_{\nu^{\ast}}\left[n^{-1}\sum_{i\in a}P_{i}(t)\widetilde{\boldsymbol{X}}_{i}(t)\right]=\boldsymbol{\mu}_{X}(t),\qquad\forall t\in\left[0,T\right].
\]

We recall that $D(\nu\parallel\upsilon)$ is the $I$-divergence or relative divergence or Kullbach-Leibler information divergence between a feasible finite measure $\nu$ with respect to a given prior measure $\upsilon$ (see, for instance, \citet{Csiszar-75}) defined by
\[
D(\nu\parallel\upsilon)=
\begin{cases}
\int_{\Omega}\log\Bigl(\frac{d\nu}{d\upsilon}\Bigr)d\nu-\nu(\Omega)+1 & \text{if }\nu\ll\upsilon\\
+\infty & \text{otherwise.}
\end{cases}
\]

To establish the connection between calibration and maximum entropy on the mean approaches the following notation %, based on \citet{Gamboa-Loubes-Rochet-11},
is adopted $\widetilde{Y}_{i}(t)=N^{-1}nd_{i}Y_{i}(t)$, $\widetilde{\boldsymbol{X}}_{i}(t)=N^{-1}nd_{i}\boldsymbol{X}_{i}(t)$
and $p_{i}(t)=\pi_{i}w_{i}(t)$, such that the functional Horvitz-Thompson estimator of $\mu_{Y}(t)$ and the functional calibration constrain~\eqref{equ1} can be, respectively, expressed as 
\[
\hat{\mu}_{Y}^{HT}(t)=N^{-1}\sum_{i\in a}d_{i}Y_{i}(t)=n^{-1}\sum_{i\in a}\widetilde{Y}_{i}(t)
\]
and 
\[
n^{-1}\sum_{i\in a}p_{i}(t)\widetilde{\boldsymbol{X}}_{i}(t)=N^{-1}\sum_{i\in a}w_{i}(t)\boldsymbol{X}_{i}(t)=\boldsymbol{\mu}_{X}(t),\qquad\forall t\in\left[0,T\right].
\]

Hence, the functional calibration estimation problem follows the structure of the maximum entropy on the mean principle, where the corresponding estimator is defined by
\[
\hat{\mu}_{Y}^{MEM}(t)=n^{-1}\sum_{i\in a}\hat{p}_{i}(t)\widetilde{Y}_{i}(t)=N^{-1}\sum_{i\in a}\hat{w}_{i}(t)Y_{i}(t).
\]

The functional calibration weighting vector $\hat{\boldsymbol{p}}(t)$ with coordinates $\hat{p}_{i}(t)=\pi_{i}\hat{w}_{i}(t)$ for $i\in a$, is the expectation of the $n$-dimensional stochastic process $\boldsymbol{P}(t)$ with coordinates $P_{i}(t)=\pi_{i}W_{i}(t)$, drawn from $\nu^{\ast}$,
\[
\hat{\boldsymbol{p}}(t)=\mathbb{E}_{\nu^{\ast}}\bigl[\boldsymbol{P}(t)\bigr],\qquad\forall t\in\left[0,T\right],
\]
where the posterior measure $\nu^{\ast}=\otimes_{i\in a}\nu_{i}^{\ast}$ (by the independence of $P_{i}$'s) maximizes the entropy $S(\cdot\parallel\upsilon)$ subject to the calibration constraint is fulfilled in mean,
\[
\mathbb{E}_{\nu^{\ast}}\left[n^{-1}\sum_{i\in a}P_{i}(t)\widetilde{\boldsymbol{X}}_{i}(t)\right]=\mathbb{E}_{\nu^{\ast}}\left[N^{-1}\sum_{i\in a}W_{i}(t)\boldsymbol{X}_{i}(t)\right]=\boldsymbol{\mu}_{X}(t),\quad\forall t\in\left[0,T\right].
\]

Note that as $p_{i}(t)=\pi_{i}w_{i}(t)$ and $\hat{w}_{i}(t)$ must to be sufficiently close to $d_{i}$, then the $\hat{p}_{i}(t)$ must be close enough to 1 for each $i\in a$.

\subsection{Reconstruction of the measure $\nu^{\ast}$}\label{subsection3.1}

For simplicity and without loss generality we assume that $T=1$. The posterior distribution $\nu^{\ast}$ can be reconstructed adopting the random measure approach for infinite dimensional inverse problems explained in detail by \citet{Gzyl-Velasquez-11}. To do this, we express the calibration constraint~\eqref{equ1} as an infinite dimensional linear inverse problem writing $w_{i}(t)$ as
\[
w_{i}(t)=\int_{0}^{1}K(s,t)\varpi_{i}\left(s\right)ds+d_{i}\quad\text{for each } i\in a,
\]
where $K(s,t)$ is a known continuous, real-valued and bounded kernel function and $\varpi_{i}=\mathbb{E}_{\nu}\left[\mathcal{W}_{i}\left(s\right)\right]$, where $\mathcal{W}$ is a stochastic process. %Note that, as $p_{i}(t)=\pi_{i}w_{i}(t)$ then $p_{i}(t)=\pi_{i}\int_{0}^{1}K(s,t)\varpi_{i}\left(s\right)ds+1$.

Hence, the infinite dimensional inverse problem, which takes the form of a Fredholm integral equation of the first kind, is 
\begin{equation}\label{restriction}
\begin{split}
\mathbb{E}_{\nu}\left[\mathcal{K}\mathcal{W}\right]
&=\mathbb{E}_{\nu}\left\lbrace\sum_{i\in a}\left[\int_{0}^{1}K(s,t)\mathcal{W}_{i}\left(s\right)ds+d_{i}\right]\boldsymbol{X}_{i}(t)\right\rbrace\\
&=\int_{0}^{1}\sum_{i\in a}K(s,t)\boldsymbol{X}_{i}(t)\varpi_{i}\left(s\right)ds+\sum_{i\in a}d_{i}\boldsymbol{X}_{i}(t)\\
&=N\boldsymbol{\mu}_{X}(t),\qquad t\in\left[0,1\right].
\end{split}
\end{equation}

To obtain the functions $\varpi_{i}^{\ast}\left(s\right)$ that solve the integral equation $\mathbb{E}_{\nu}\left[\mathcal{K}\mathcal{W}\right]=N\boldsymbol{\mu}_{X}(t)$, the random measure approach adopted considers $\varpi_{i}\left(s\right)$ as a density of a measure $\varpi_{i}\left(s\right)ds$, $i\in a$. Under this setting, we define the random measure $\mathcal{W}_{i}\left(a,b\right]=\mathcal{W}_{i}(b)-\mathcal{W}_{i}(a)$ for $\left(a,b\right]\subset\left[0,1\right]$ such that $dE_{\nu}\left\lbrace \mathcal{W}_{i}\left(0,s\right]\right\rbrace =\varpi_{i}(s)ds$ for each $i\in a$. The next theorem ensures the existence of the posterior distribution $\nu^{\ast}$ to obtain the functions $\varpi_{i}^{\ast}\left(s\right)$ depending on the assumed prior distribution $\upsilon$.

\begin{theorem}\label{theorem.2} Let $\upsilon$ be a prior positive probability measure, $\boldsymbol{\lambda}=\boldsymbol{\lambda}(t)$ a measure in the class of continuous measures on $\left[0,1\right]^{q}$, $\mathcal{M}\left(C\left[0,1\right]^{q}\right)$, and $\mathcal{V}=\left\lbrace \nu\ll\upsilon\colon Z_{\upsilon}(\boldsymbol{\lambda})<+\infty\right\rbrace$ a nonempty open class, where $Z_{\upsilon}(\boldsymbol{\lambda})=\mathbb{E}_{\upsilon}\left[\exp\left\lbrace \langle\boldsymbol{\lambda},\mathcal{K}\mathcal{W}\rangle\right\rbrace \right]$, with
\begin{equation}\label{innerproduct}
\left\langle \boldsymbol{\lambda},\mathcal{KW}\right\rangle =\int_{0}^{1}\boldsymbol{\lambda}^{\top}(dt)\left(\int_{0}^{1}\sum_{i\in a}K(s,t)\boldsymbol{X}_{i}(t)d\mathcal{W}_{i}(s)+\sum_{i\in a}d_{i}\boldsymbol{X}_{i}(t)\right).
\end{equation}
Then there exists an unique probability measure 
\[
\nu^{\ast}=\argmax_{\mathcal{\nu\in V}}S(\nu\parallel\upsilon),
\]
subject to $E_{\nu}\left[\mathcal{K}\mathcal{W}\right]=N\boldsymbol{\mu}_{X}(t)$,
which is achieved at
\[
d\nu^{\ast}/d\upsilon=Z_{\upsilon}^{-1}(\boldsymbol{\lambda}^{\ast})\exp\left\lbrace \langle\boldsymbol{\lambda}^{\ast},\mathcal{KW}\rangle\right\rbrace,
\]
where $\boldsymbol{\lambda}^{\ast}(t)$ minimizes the functional 
\[
H_{\upsilon}(\boldsymbol{\lambda})=\log Z_{\upsilon}(\boldsymbol{\lambda})-\langle\boldsymbol{\lambda},N\boldsymbol{\mu}_{X}\rangle.
\]
\end{theorem}

Based on the Theorem \ref{theorem.2}, we will carry out the reconstruction of $\nu$, assuming the centered Gaussian and compound Poisson random measures as prior measures, in order to estimate the respective functional calibration weights $\hat{w}_{i}(t)$, $i\in a$. The estimates are given by the following two Lemmas.

\begin{lemma}\label{lemma.1} Let $\upsilon$ be a centered stationary Gaussian measure on $\bigl(\mathcal{C}(\left[0,1\right]),\mathcal{B}(\mathcal{C}(\left[0,1\right]))\bigr)$, and $\boldsymbol{\lambda}=\boldsymbol{\lambda}(t)\in\mathcal{M}\left(C\left[0,1\right]^{q}\right)$. Then, $\hat{w}_{i}(t)=\int_{0}^{1}K(s,t)\varpi^{\ast}(s)ds+d_{i}$ $i\in a$, where
\[
\varpi^{\ast}(s)=\sum_{i'\in a}\int_{0}^{1}K(s,t')\boldsymbol{X}_{i'}^{\top}(t')\boldsymbol{\lambda}^{\ast}(dt').
\]
\end{lemma}

\begin{lemma}\label{lemma.2} Let $\mathcal{W}_{i}(s)=\sum_{k=1}^{N(s)}\xi_{ik}$ be a compound Poisson process, where $N(s)$ is a homogeneous Poisson process on $\left[0,1\right]$ with intensity parameter $\gamma>0$, and $\xi_{ik}$, $k\geq 1$ are independent and identically distributed real-valued random variables for each $i\in a$ with distribution $u$ on $\mathbb{R}$ satisfying $u(\left\{ 0\right\} )=0$, and independent of $N(s)$. %Further, let $\boldsymbol{\lambda}=\boldsymbol{\lambda}(t)\in\mathcal{M}\left(C\left[0,1\right]^{q}\right)$.
Then, $\hat{w}_{i}(t)=\int_{0}^{1}K(s,t)\varpi^{\ast}(s)ds+d_{i}$ $i\in a$, where
\[
\varpi_{i}^{\ast}(s)=\int_{\mathbb{R}}\xi_{i}\exp\left\{ \sum_{i\in a}\int_{0}^{1}K(s,t)\xi_{i}\boldsymbol{X}_{i}^{\top}(t)\boldsymbol{\lambda}^{\ast}(dt)\right\} u\left(d\xi_{i}\right).
\]
\end{lemma}

\section{Approximation of the maximum entropy on the mean functional estimator}
\label{sec:4}

To approximate the functional calibration weights and the functional maximum entropy on the mean estimator for the finite population mean of $Y(t)$ with the assumed prior measure, an Euler discretization scheme is used. Consider a partition of $(s,t)\in\left[0,1\right]^{2}$ in ${J}$ and $L$ equidistant fixed points, $(j-1)/J<s_{j}\leq j/J$, $j=1,\ldots,J$, $(l-1)/L<t_{l}\leq l/L$, $l=1,\ldots,L$, respectively. For the corresponding prior measures, the approximations for functions $Z_{\upsilon}(\boldsymbol{\lambda})$, $H_{\upsilon}(\boldsymbol{\lambda})$ and $\boldsymbol{\lambda}^{\ast}(t)$ are based on the respective results found in the Appendix.

\subsection{Centered Gaussian measure}

For a prior centered Gaussian random measure, the approximations of the linear moment calibration constraint~\eqref{restriction} and the inner product $\langle\boldsymbol{\lambda},\mathcal{K}\mathcal{W}\rangle$ are, respectively, given by 
\[
\mathbb{E}_{\nu}\left[\sum_{j=1}^{J}\sum_{i\in a}K(s_{j},t_{l})\Delta\mathcal{W}_{i}(s_{j})\boldsymbol{X}_{i}(t_{l})+\sum_{i\in a}d_{i}\boldsymbol{X}_{i}(t_{l})\right]=N\boldsymbol{\mu}_{X}(t_{l})
\]
and 
\begin{align*}
 & \frac{1}{L}\sum_{l=1}^{L}\boldsymbol{\lambda}^{\top}(t_{l})\sum_{j=1}^{J}\sum_{i\in a}K(s_{j},t_{l})\Delta\mathcal{W}_{i}(s_{j})\boldsymbol{X}_{i}(t_{l})+\frac{1}{L}\sum_{l=1}^{L}\boldsymbol{\lambda}^{\top}(t_{l})\sum_{i\in a}d_{i}\boldsymbol{X}_{i}(t_{l})\\
 & =\frac{1}{L}\sum_{j=1}^{J}\sum_{i\in a}\sum_{l=1}^{L}K(s_{j},t_{l})\Delta\mathcal{W}_{i}(s_{j})\boldsymbol{\lambda}^{\top}(t_{l})\boldsymbol{X}_{i}(t_{l})+\frac{1}{L}\sum_{i\in a}d_{i}\sum_{l=1}^{L}\boldsymbol{\lambda}^{\top}(t_{l})\boldsymbol{X}_{i}(t_{l}),
\end{align*}
where $\Delta\mathcal{W}_{i}(s_{j})=\mathcal{W}_{i}(s_{j})-\mathcal{W}_{i}(s_{j-1})$
is the discrete version of $d\mathcal{W}_{i}(s)$ for $i\in a$.

Therefore, we have that $Z_{\upsilon}(\boldsymbol{\lambda})$ is approximated at the grid (see equation~\eqref{equ2} of the proof of Lemma 1 in the Appendix) by
\[
\begin{split}
& \mathbb{E}_{\upsilon}\left[\exp\left\lbrace \frac{1}{L}\sum_{i\in a}d_{i}\sum_{l=1}^{L}\boldsymbol{\lambda}^{\top}(t_{l})\boldsymbol{X}_{i}(t_{l})+\frac{1}{L}\sum_{j=1}^{J}\sum_{i\in a}\sum_{l=1}^{L}K(s_{j},t_{l})\boldsymbol{\lambda}^{\top}(t_{l})\boldsymbol{X}_{i}(t_{l})\Delta\mathcal{W}_{i}(s_{j})\right\rbrace \right]\\
 & =\exp\left\{ \frac{1}{L}\sum_{i\in a}d_{i}\sum_{l=1}^{L}\boldsymbol{\lambda}^{\top}(t_{l})\boldsymbol{X}_{i}(t_{l})+\sum_{j=1}^{J}\frac{1}{2J}\left(\frac{1}{L}\sum_{i\in a}\sum_{l=1}^{L}K(s_{j},t_{l})\boldsymbol{\lambda}^{\top}(t_{l})\boldsymbol{X}_{i}(t_{l})\right)^{2}\right\} \\
% & =\exp\left\{ \frac{1}{L}\sum_{i\in a}d_{i}\sum_{l=1}^{L}\boldsymbol{\lambda}^{\top}(t_{l})\boldsymbol{X}_{i}(t_{l})\right\} \prod_{j=1}^{J}\exp\left\{ \frac{1}{2J}\left(\frac{1}{L}\sum_{i\in a}\sum_{l=1}^{L}K(s_{j},t_{l})\boldsymbol{\lambda}^{\top}(t_{l})\boldsymbol{X}_{i}(t_{l})\right)^{2}\right\} \\
% & =\exp\left\{ \frac{1}{L}\sum_{i\in a}d_{i}\sum_{l=1}^{L}\boldsymbol{\lambda}^{\top}(t_{l})\boldsymbol{X}_{i}(t_{l})\right\} \\
% & \quad\times\prod_{j=1}^{J}\exp\left\{ \frac{1}{2J}\left(\frac{1}{L}\sum_{i\in a}\sum_{l=1}^{L}K(s_{j},t_{l})\boldsymbol{\lambda}^{\top}(t_{l})\boldsymbol{X}_{i}(t_{l})\right)\left(\frac{1}{L}\sum_{i'\in a}\sum_{l=1}^{L}K(s_{j},t'_{l})\boldsymbol{\lambda}^{\top}(t'_{l})\boldsymbol{X}_{i'}(t'_{l})\right)\right\} \\
 & =\exp\left\{ \frac{1}{L}\sum_{i\in a}d_{i}\sum_{l=1}^{L}\boldsymbol{\lambda}^{\top}(t_{l})\boldsymbol{X}_{i}(t_{l})\right\} \prod_{j=1}^{J}\exp\left\{ \frac{1}{2J}\sum_{i\in a}\sum_{i'\in a}h_{i}(s_{j})h_{i'}(s_{j})\right\} \\
 & =\exp\left\{ \frac{1}{L}\sum_{i\in a}d_{i}\sum_{l=1}^{L}\boldsymbol{\lambda}^{\top}(t_{l})\boldsymbol{X}_{i}(t_{l})\right\} \prod_{j=1}^{J}z_{i}\left(h_{i}(s_{j})\right),
\end{split}
\]
where $h_{i}(s_{j})=L^{-1}\sum_{l=1}^{L}K(s_{j},t_{l})\boldsymbol{\lambda}^{\top}(t_{l})\boldsymbol{X}_{i}(t_{l})$,
$i\in a$, $j=1,\ldots J$, and $l=1,\ldots L$.

Now, the finite dimensional maxentropic solution for $\varpi_{i}(s_{j})$ for each $i\in a$ is approximated by (see \citet{Gzyl-Velasquez-11})
\begin{equation}\label{varpi.gauss}
\begin{split}
\varpi_{i}^{\ast}(s_{j})
& =\left.\frac{\text{d}\log z_{i}\left(h_{i}(s_{j})\right)}{\text{d}(2J)^{-1}h_{i}(s_{j})}\right|_{h_{i}(s_{j})=\mathcal{K}\boldsymbol{\lambda}^{\ast}}\\
 & =\left.\sum_{i'\in a}h_{i'}(s_{j})\right|_{h_{i}(s_{j})=\mathcal{K}\boldsymbol{\lambda}^{\ast}}\\
 & =\frac{1}{L}\sum_{l=1}^{L}\sum_{i'\in a}K(s_{j},t'_{l})\boldsymbol{\lambda}^{\ast\top}(t'_{l})\boldsymbol{X}_{i'}(t'_{l}),
\end{split}
\end{equation}
where the finite dimensional version of $\boldsymbol{\lambda}^{\ast}(t_{l}')$,
$(l-1)/L<t_{l}\leq l/L$, $l=1,\ldots,L$, is the minimizer of $H_{\upsilon}(\boldsymbol{\lambda})$,
whose approximation (see equation~\eqref{equ3} of the proof of Lemma 1 in the Appendix) is
\[
\begin{split} & \frac{1}{2}\sum_{l=1}^{L}\sum_{l=1}^{L}\boldsymbol{\lambda}^{\top}(t_{l})\left(\frac{1}{JL^{2}}\sum_{j=1}^{J}K(s_{j},t_{l})K(s_{j},t_{l}')\sum_{i\in a}\sum_{i'\in a}\boldsymbol{X}_{i}(t_{l})\boldsymbol{X}_{i'}^{\top}(t_{l}')\right)\boldsymbol{\lambda}(t_{l}')\\
 & \quad+\frac{1}{L}\sum_{l=1}^{L}\left(\sum_{i\in a}d_{i}\boldsymbol{X}_{i}^{\top}(t_{l})-N\boldsymbol{\mu}_{X}^{\top}(t_{l})\right)\boldsymbol{\lambda}(t_{l}).
\end{split}
\]
The first order condition (see equation\eqref{equ4}) associated to this minimization problem is 
\[
\begin{split}
\frac{1}{JL^{2}}\sum_{j=1}^{J}\sum_{l=1}^{L}K(s_{j},t_{l})K(s_{j},t_{l}')\sum_{i\in a}\sum_{i'\in a}\boldsymbol{X}_{i}(t_{l})\boldsymbol{X}_{i'}^{\top}(t_{l}')\boldsymbol{\lambda}^{\ast}(t_{l}')\\
+\frac{1}{L}\left(N\boldsymbol{\mu}_{X}(t_{l})-\sum_{i\in a}d_{i}\boldsymbol{X}_{i}(t_{l})\right) & =\boldsymbol{0},
\end{split}
\]
\noindent whose solution $\boldsymbol{\lambda}^{\ast}(t_{l}')$ is given by
\[
\begin{split}
\boldsymbol{\lambda^{\ast}}(t_{l}')  %&=\left(\frac{1}{JL^{2}}\sum_{j=1}^{J}\sum_{l=1}^{L}K(s_{j},t_{l})K(s_{j},t_{l}')\sum_{i\in a}\sum_{i'\in a}\boldsymbol{X}_{i}(t_{l})\boldsymbol{X}_{i'}^{\top}(t_{l}')\right)^{-1}\\
% & \quad\times\frac{1}{L}\left(N\boldsymbol{\mu}_{X}(t_{l})-\sum_{i\in a}d_{i}\boldsymbol{X}_{i}(t_{l})\right)\\
 & =JL\left(\sum_{j=1}^{J}\sum_{l=1}^{L}K(s_{j},t_{l})K(s_{j},t_{l}')\sum_{i\in a}\sum_{i'\in a}\boldsymbol{X}_{i}(t_{l})\boldsymbol{X}_{i'}^{\top}(t_{l}')\right)^{-1}\\
 & \quad\;\;\,\times\left(N\boldsymbol{\mu}_{X}(t_{l})-\sum_{i\in a}d_{i}\boldsymbol{X}_{i}(t_{l})\right).
\end{split}
\]
Finally, the approximation of the finite dimensional solution of $\hat{w}_{i}(t)$ is 
\[
\begin{split}
\hat{w}_{i}(t_{l})
& =\frac{1}{J}\sum_{j=1}^{J}K(s_{j},t_{l})\varpi_{i}^{\ast}(s_{j})+ d_{i},
\end{split}
\]
where $\varpi_{i}^{\ast}(s_{j})$ es given by the equation~\eqref{varpi.gauss}.

\subsection{Compound Poisson measure}

Based on equations~\eqref{equ5} and~\eqref{equ6} of the proof of Lemma 2 in the Appendix, the approximation of $Z_{\upsilon}(\boldsymbol{\lambda})$ is given by 
\begin{align*}
 & \mathbb{E}_{\upsilon}\left[\exp\left\{ \langle g(s_{j}),d\mathcal{W}_{i}\rangle+\left\langle \boldsymbol{\lambda},\sum_{i\in a}d_{i}\boldsymbol{X}_{i}(t_{l})\right\rangle \right\} \right]\\
 & =\exp\left\{ \left\langle \boldsymbol{\lambda},\sum_{i\in a}d_{i}\boldsymbol{X}_{i}(t_{l})\right\rangle \right\} \mathbb{E}_{\upsilon}\left[\exp\left\{ \langle g(s_{j}),d\mathcal{W}_{i}\rangle\right\} \right]\\
% & =\exp\left\{ \left\langle \boldsymbol{\lambda},\sum_{i\in a}d_{i}\boldsymbol{X}_{i}(t_{l})\right\rangle \right\} \prod_{j=1}^{J}\mathbb{E}_{\upsilon}\left[\exp\left\{ g\left(s_{j}\right)m_{i}\left(\left(s_{j-1},s_{j}\right]\right)\right\} \right]\\
% & =\exp\left\{ \left\langle \boldsymbol{\lambda},\sum_{i\in a}d_{i}\boldsymbol{X}_{i}(t_{l})\right\rangle \right\} \prod_{j=1}^{J}\exp\left\{ \mathbb{E}_{\upsilon}\left[\exp\left\{ g\left(s_{j}\right)\xi_{i}\right\} \right]\right\} \\
% & =\exp\left\{ \left\langle \boldsymbol{\lambda},\sum_{i\in a}d_{i}\boldsymbol{X}_{i}(t_{l})\right\rangle \right\} \prod_{j=1}^{J}\exp\left\{ \frac{\gamma}{J}\int_{\mathbb{R}}\left(\exp\left\{ g\left(s_{j}\right)\xi_{i}\right\} -1\right)u\left(d\xi_{i}\right)\right\} \\
 & =\exp\left\{ \left\langle \boldsymbol{\lambda},\sum_{i\in a}d_{i}\boldsymbol{X}_{i}(t_{l})\right\rangle \right\} \\
 & \quad\times\prod_{j=1}^{J}\exp\left\{ \frac{\gamma}{J}\int_{\mathbb{R}}\left(\exp\left\{ \frac{1}{L}\sum_{i\in a}\xi_{i}\sum_{l=1}^{L}K(s_{j},t_{l})\boldsymbol{\lambda}^{\top}(t_{l})\boldsymbol{X}_{i}(t_{l})\right\} -1\right)u\left(d\xi_{i}\right)\right\} \\
 & =\exp\left\{ \left\langle \boldsymbol{\lambda},\sum_{i\in a}d_{i}\boldsymbol{X}_{i}(t_{l})\right\rangle \right\} \prod_{j=1}^{J}\exp\left\{ \frac{\gamma}{J}\int_{\mathbb{R}}\left(\exp\left\{ \sum_{i\in a}\xi_{i}h_{i}(s_{j})\right\} -1\right)u\left(d\xi_{i}\right)\right\} \\
 & =\exp\left\{ \left\langle \boldsymbol{\lambda},\sum_{i\in a}d_{i}\boldsymbol{X}_{i}(t_{l})\right\rangle \right\} \prod_{j=1}^{J}z_{i}\left(h_{i}(s_{j})\right),\qquad i\in a,
\end{align*}
where $h_{i}(s_{j})=L^{-1}\sum_{l=1}^{L}K(s_{j},t_{l})\boldsymbol{\lambda}^{\top}(t_{l})\boldsymbol{X}_{i}(t)$, $i\in a$, $j=1,\ldots J$, and $\left\langle \boldsymbol{\lambda},\sum_{i\in a}d_{i}\boldsymbol{X}_{i}(t_{l})\right\rangle =L^{-1}\sum_{i\in a}d_{i}\sum_{l=1}^{L}\boldsymbol{\lambda}^{\top}(t_{l})\boldsymbol{X}_{i}(t_{l})$.

\noindent The approximated maxentropic solution for $\varpi_{i}(s_{j})$ for each $i\in a$ is 
\begin{equation}\label{varpi.poisson}
\begin{split}
\varpi_{i}^{\ast}(s_{j}) & =\left.\frac{\text{d}\log z_{i}\left(h_{i}(s_{j})\right)}{\text{d}h_{i}(s_{j})}\right|_{h_{i}(s_{j})=\mathcal{K}\boldsymbol{\lambda}^{\ast}}\\
 & =\left.\frac{\gamma}{J}\int_{\mathbb{R}}\xi_{i}\exp\left\{ \sum_{i\in a}\xi_{i}h_{i}(s_{j})\right\} u\left(d\xi_{i}\right)\right|_{h_{i}(s_{j})=\mathcal{K}\boldsymbol{\lambda}^{\ast}}\\
 & =\frac{\gamma}{J}\int_{\mathbb{R}}\xi_{i}\exp\left\{ \frac{1}{L}\sum_{i\in a}\sum_{l=1}^{L}K(s_{j},t_{l})\xi_{i}\boldsymbol{X}_{i}^{\top}(t_{l})\boldsymbol{\lambda}^{\ast}(t_{l})\right\} u\left(d\xi_{i}\right),
\end{split}
\end{equation}
where the finite dimensional version of $\boldsymbol{\lambda}^{\ast}(t_{l})$,
is the minimizer of $H_{\upsilon}(\boldsymbol{\lambda})$, whose approximation, by the equation~\eqref{equ7} of the proof of Lemma 2 in the Appendix, is
\[
\begin{split}H_{\upsilon}(\boldsymbol{\lambda})
%& =\log Z_{\upsilon}(\boldsymbol{\lambda})-\langle\boldsymbol{\lambda},N\boldsymbol{\mu}_{X}\rangle\\
 & =\frac{\gamma}{J}\sum_{j=1}^{J}\int_{\mathbb{R}}\left(\exp\left\{ \frac{1}{L}\sum_{i\in a}\sum_{l=1}^{L}K(s_{j},t_{l})\xi_{i}\boldsymbol{X}_{i}^{\top}(t_{l})\boldsymbol{\lambda}(t_{l})\right\} -1\right)u\left(d\xi_{i}\right)\\
 & \quad+\frac{1}{L}\sum_{l=1}^{L}\left(\sum_{i\in a}d_{i}\boldsymbol{X}_{i}^{\top}(t_{l})-N\boldsymbol{\mu}_{X}^{\top}(t_{l})\right)\boldsymbol{\lambda}(t_{l})
\end{split}
\]
The corresponding equation for $\boldsymbol{\lambda}^{\ast}(t_{l})$ that minimizes $H_{\upsilon}(\boldsymbol{\lambda})$ is given by the nonlinear system of equations (see equation~\eqref{equ8} in the Appendix)
\begin{align*}
\sum_{i\in a}\left[\frac{1}{J}\sum_{j=1}^{J}K(s_{j},t_{l})\left(\gamma L\int_{\mathbb{R}}\xi_{i}\exp\left\{ \frac{1}{L}\sum_{i\in a}\sum_{l=1}^{L}K(s_{j},t_{l})\xi_{i}\boldsymbol{X}_{i}^{\top}(t_{l})\boldsymbol{\lambda}^{\ast}(t_{l})\right\} u\left(d\xi_{i}\right)\right)+d_{i}\right]\\
\times\boldsymbol{X}_{i}(t_{l})=N\boldsymbol{\mu}_{X}(t_{l})
\end{align*}
Finally, as in the Gaussian measure case, the finite dimensional solution of $\hat{w}_{i}(t)$ is approximated by $\hat{w}_{i}(t_{l})=J^{-1}\sum_{j=1}^{J}K(s_{j},t_{l})\varpi_{i}^{\ast}(s_{j}) + d_{i}$ with $\varpi_{i}^{\ast}(s_{j})$ given by the equation~\eqref{varpi.poisson}.

\section{Simulation study}
\label{sec:5}

We shall illustrate through a simple simulation study the performance of results obtained in the above section. Considering a finite population $U_{N}$ of size $N=1000$, we generate a functional random variable
$Y_{i}(t)$ by the point-wise multiple linear regression model 
\[
Y_{i}(t)=\alpha(t)+\boldsymbol{X}_{i}(t)^{\top}\boldsymbol{\beta}(t)+\varepsilon_{i}(t),\qquad i\in U_{N},
\]
where $\alpha(t)=1.2+2.3\cos\left(2\pi t\right)+4.2\sin\left(2\pi t\right)$,
$\boldsymbol{\beta}(t)=\left(\beta_{1}(t),\beta_{2}(t)\right)^{\top}$
with $\beta_{1}(t)=\cos\left(10t\right)$ and $\beta_{2}(t)=t\sin\left(15t\right)$,
$\boldsymbol{X}_{i}(t)=\left(X_{i1}(t),X_{i2}(t)\right)^{\top}$,
and $\varepsilon_{i}(t)\sim\mathcal{N}\left(0,\sigma_{\varepsilon}^{2}(1+t)\right)$
with $\sigma_{\varepsilon}^{2}=0.1$, and independent of $\boldsymbol{X}_{i}(t)$.
The auxiliary functional covariates are defined by $X_{i1}(t)=\mathcal{U}_{i1}+f_{1}(t)$
with $f_{1}(t)=3\sin(3\pi t+3)$, and $X_{i2}(t)=\mathcal{U}_{i2}+f_{2}(t)$
with $f_{2}(t)=-\cos(\pi t)$, where $\mathcal{U}_{i1}$ and $\mathcal{U}_{i2}$
are independent and, respectively, i.i.d. uniform random variables
on the intervals $\left[-1,1.3\right]$ and $\left[-0.5,0.5\right]$.
The design time points for $t\in\left[0,1\right]$ and $s\in\left[0,1\right]$
are $t_{j}=j/J$, $j=1,\ldots,J$ and $s_{l}=l/L$, $l=1,\ldots,L$,
with $J=50$ and $L=80$ The Figures~\ref{fig.1} and \ref{fig.2}
show, respectively, the simulated finite population auxiliary functional
covariates and functional responses for each $i\in U_{N}$, and the
respective finite population functional means, $\boldsymbol{\mu}_{X}(t)=\left(\mu_{X_{1}}(t),\mu_{X_{2}}(t)\right)^{\top}$
and $\mu_{Y}(t)=N^{-1}\sum_{i\in U_{N}}Y_{i}(t)$. Assuming a uniform
fixed-size sampling design we drawn a sample $a\in U_{N}$ of $n=0.12N$
elements without replacement. For the kernel function we assumed a
Gaussian one, $K(t,s)=\exp\left\{ -\left|t-s\right|^{2}/2\sigma^{2}\right\} $
with $\sigma^{2}=0.5$. The random variables $\xi_{i}$ for the compound
Poisson case are assumed i.i.d. uniform on the interval $\left[-1,1\right]$,
and $\gamma=1$. To solve the nonlinear system of equations for $\boldsymbol{\lambda}^{\ast}(t_{l})$
in the compound Poisson case, we used the \texttt{R}-package \texttt{BB}
(see \citet{Varadhan-12} and \citet{Varadhan-Gilbert-09}).

\begin{figure}[!h]
	\centering
	\includegraphics[height=8cm,width=14cm]{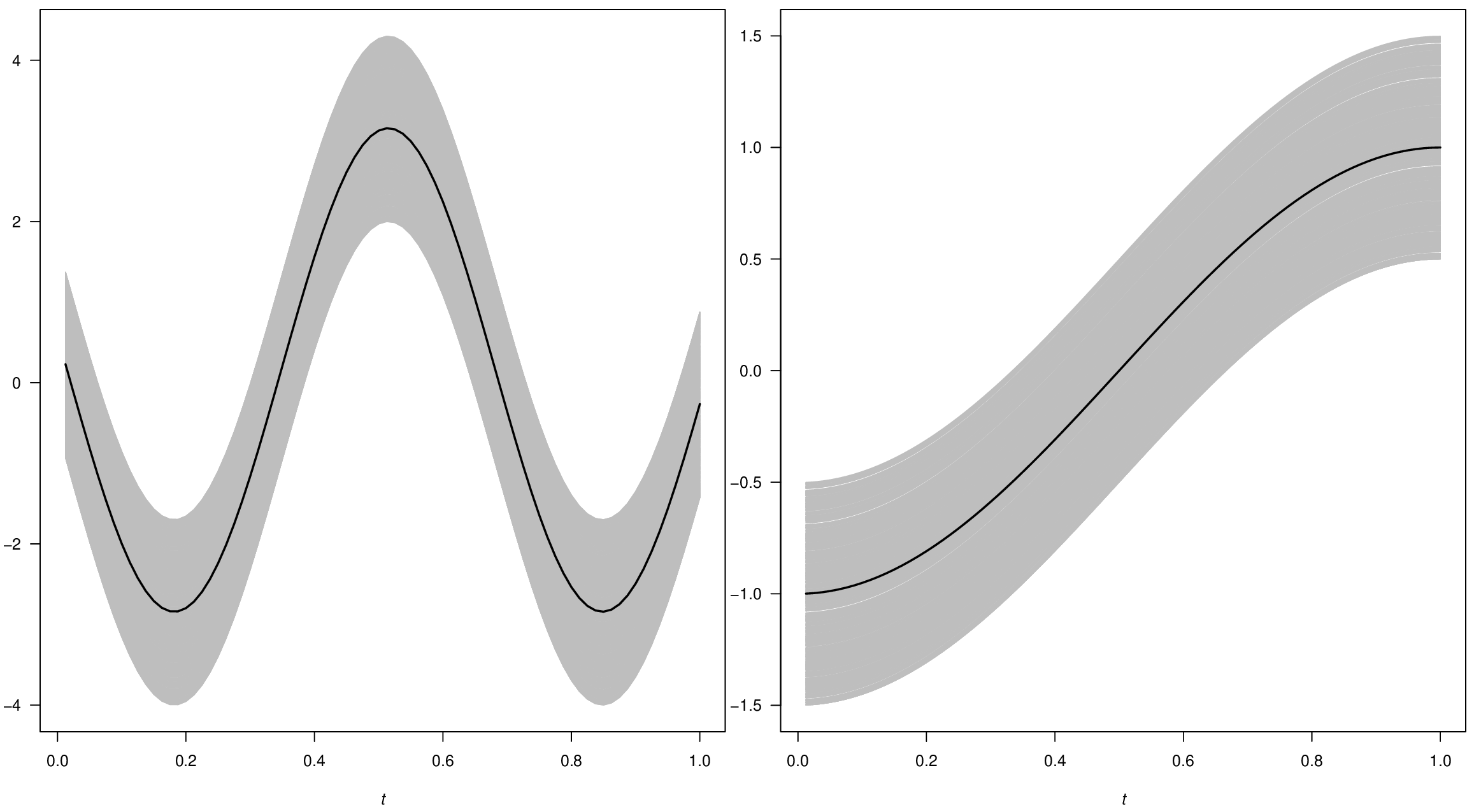}
	\caption{Population auxiliary functional variables (gray), $X_{i1}(t)$ (on the left) and $X_{i2}(t)$ (on the right). Functional finite population means, $\mu_{X_{1}}(t)$ and $\mu_{X_{2}}(t)$ (solid line)}
\label{fig.1}
\end{figure}

The graphical comparisons of the estimators for a random selected repetition are illustrated in the Figure~\ref{fig.2}. The figure shows, in general, a good performance, specially for the estimator
assuming the Gaussian measure. The principal differences with respect to the theoretical functional finite population mean are localized on the edges, particularly on the left edge. The Horvitz-Thompson estimator, in both cases, has a little departure localized around the deep valley. However our estimator has not this departure. A nice feature of the functional calibration method is that permits to check graphically how well the estimator satisfies the calibration constraints for each covariate, $N^{-1}\sum_{i\in a}\hat{w}_{i}(t)\boldsymbol{X}_{i}(t)=\boldsymbol{\mu}_{X}(t)$. This is illustrated in the Figure~\ref{fig.3}.
\begin{figure}
	\centering
	\includegraphics[height=8cm,width=14cm]{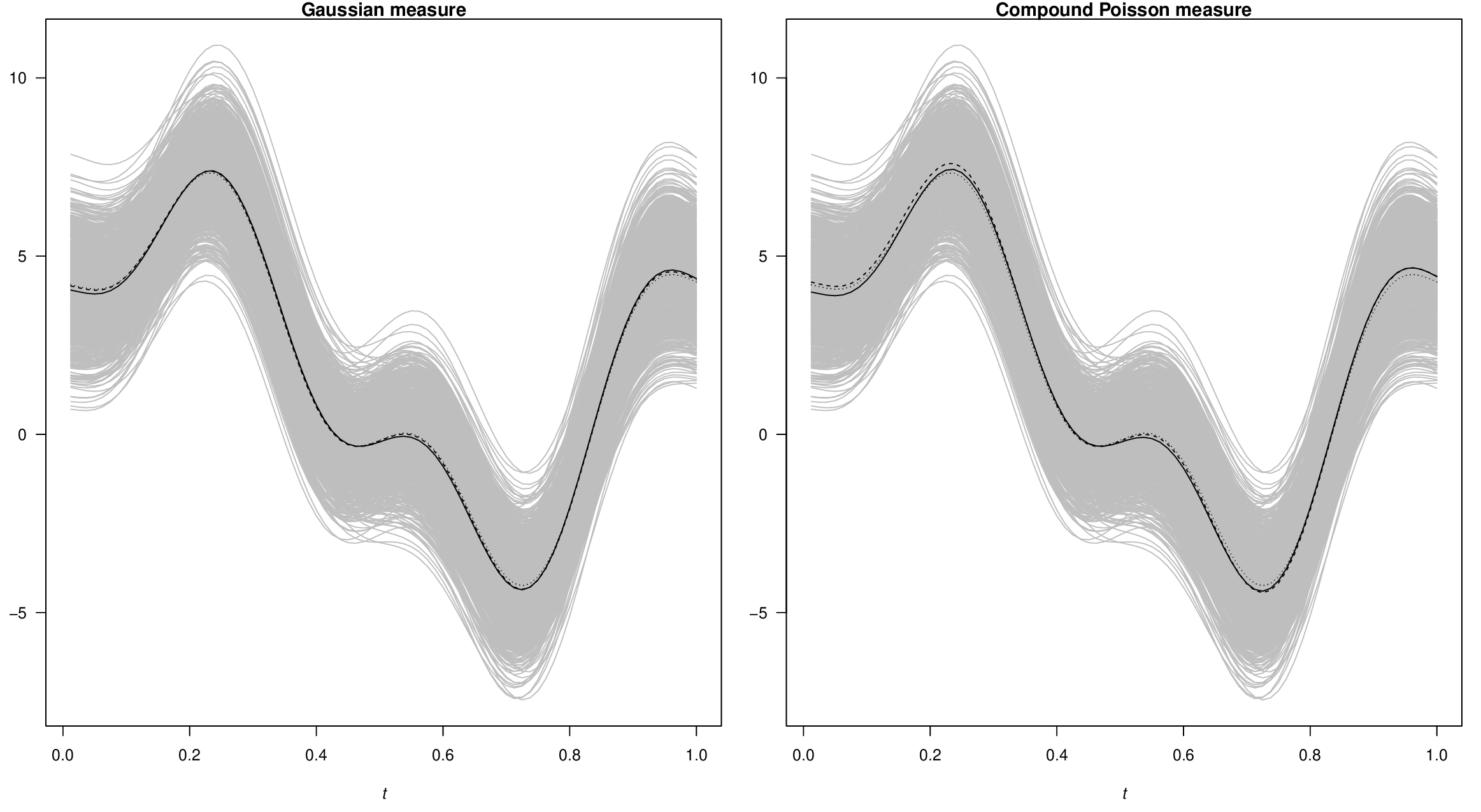}
	\caption{Population survey functions $Y_{i}(t)$ (in gray), functional finite population mean $\mu_{Y}(t)$ (solid line), and the Horvitz-Thompson (dotted line) and functional maximum entropy on the mean (dashed line) estimators}
\label{fig.2} 
\end{figure}
\begin{figure}
	\centering
	\includegraphics[height=8cm,width=14cm]{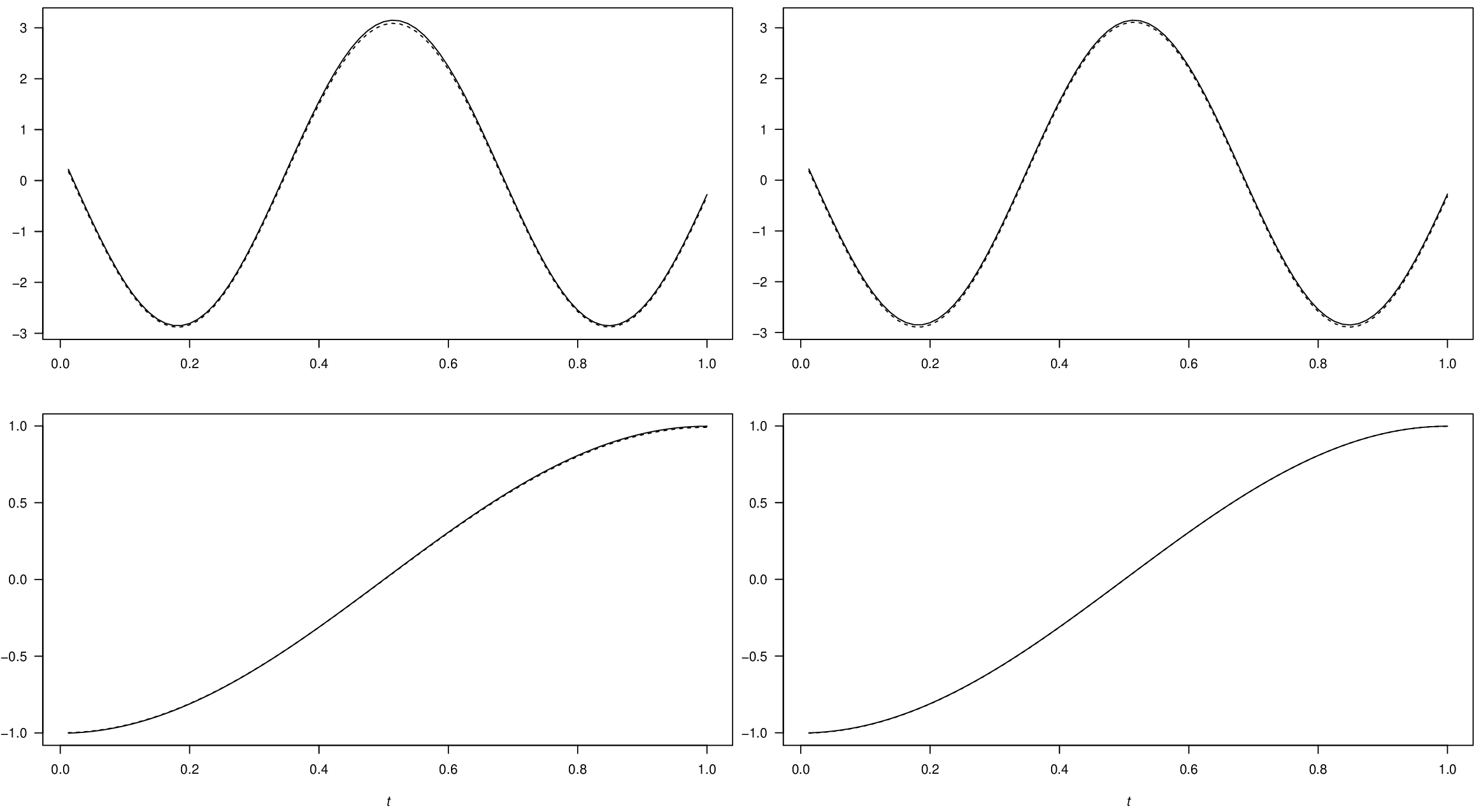}
	\caption{Functional calibration constraint~\eqref{equ1} for Gaussian (on the left) and compound Poisson (on the right) measures. $\boldsymbol{\mu}_{X}(t)$ (solid line), $N^{-1}\sum_{i\in a}\hat{w}_{i}(t)\boldsymbol{X}_{i}(t)$ (dash)}
\label{fig.3} 
\end{figure}

To evaluate the performance of the maximum entropic functional calibration estimator, $\hat{\mu}_{Y}^{MEM}(t)$, assuming the Gaussian and compound Poisson prior measures, we calculated its empirical bias\textendash{}variance decomposition of the mean square errors and compare it with the functional Horvitz-Thompson estimator $\hat{\mu}_{Y}^{HT}(t)$. The simulation study was conducted with 100 repetitions. In Table~\ref{Table1} we can see that, with respect to the Horvitz-Thompson estimator, the maximum entropic estimator has smaller variance and mean square error for both prior measures, particularly for the Gaussian prior. Although the Horvitz-Thompson estimator has smaller bias squared, the differences are not significant. Also, the small value for the bias confirm the unbiasedness of the functional maximum entropy on the mean and Horvitz-Thompson estimators.

\begin{table}
\caption{Bias-variance decomposition of MSE}
\label{Table1}
\centering{\small{
\begin{tabular}{llll} \hline \noalign{\smallskip}
Functional estimator	   			   &  MSE 	& $\text{Bias}^{2}$ & Variance\\\noalign{\smallskip}\hline\noalign{\smallskip}
Horvitz-Thompson  					   & 0.2391 & 0.0005 & 0.2386\\
Maximum entropy on the mean (Gaussian) & 0.2001 & 0.0006 & 0.1995\\
Maximum entropy on the mean (Poisson)  & 0.2333 & 0.0084 & 0.2249\\
\noalign{\smallskip}\hline
\end{tabular}}
}
\end{table}

\section{Concluding remarks}
\label{sec:6}

In this paper we have proposed an extension to the problem of obtaining an estimator for the finite population mean of a survey variable incorporating complete auxiliary information under an infinite dimensional setting. Considering that both the survey and the set of auxiliary variables are functions, the respective functional calibration constraint is expressed as an infinite dimensional linear inverse problem, whose solution offers the functional survey weights of the calibration estimator. The solution of the problem is conducted by mean the maximum entropy on the mean principle, which is a powerful probabilistic-based regularization method to solve constrained linear inverse problems. Here we assume a centered Gaussian and compound Poisson random measures as prior measures to obtain the functional calibration weights. However, other random measures can be considered also.

The simulations study results show that the proposed functional calibration estimator improves its accuracy compared with the Horvitz-Thompson estimator. In the simulations, both the functional survey and auxiliary variables where assumed with amplitude variations (variation in the $y$-axis) only. More complex extensions allowing both amplitude and phase (variation in the $x$-axis) variations are possible. %The study of the large sample statistical properties of our technique will be examined in the forthcoming paper by \citet{Gallon-Gamboa-Loubes-12}.

Finally, a further interesting extension of the functional calibration estimation problem under the maximum entropy on the mean approach can be conducted following the idea of model-calibration proposed by \citet{Wu-Sitter-01}, \citet{Wu-03} and \citet{Montanari-Ranalli-05}. This may be accomplished considering a nonparametric functional regression $Y_{i}(t)=\mu\left\lbrace \boldsymbol{X}_{i}(t)\right\rbrace +\varepsilon_{i}(t)$, $i\in U_{N}$, $t\in\bigl(\left[0,T\right]$ to model the relation between the functional survey variable and the set of functional auxiliary covariates in order to allows a more effective use of the functional auxiliary information.

%\appendix
\section*{Appendix}\label{appendix}

\begin{proof}[Proof of Theorem~\ref{theorem.1}]
%\textit{of Theorem~\ref{theorem.1}.}}\vskip .1in

\noindent The Lagrangian function associated to the restricted minimization problem is 
\[
L_{a}(\boldsymbol{w},\boldsymbol{\lambda})=\mathcal{D}_{a}^{\ast}(w,d)+\boldsymbol{\lambda^{\top}}(t)\left(\boldsymbol{\mu}_{X}(t)-N^{-1}\sum_{i\in a}w_{i}(t)\boldsymbol{X}_{i}(t)\right),
\]
where $\boldsymbol{\lambda}(t)$ is the corresponding functional Lagrange multiplier vector. The first order conditions are 
\[
\frac{w_{i}(t)-d_{i}}{d_{i}q_{i}\left(t\right)}-\boldsymbol{\lambda}(t)^{\top}\boldsymbol{X}_{i}(t)=0,\qquad i\in a
\]
which can be expressed as
\[
w_{i}(t)=d_{i}\left[1+q_{i}(t)\boldsymbol{\lambda}(t)^{\top}\boldsymbol{X}_{i}(t)\right],\qquad i\in a
\]
where, its uniqueness is guaranteed by the continuous differentiability of $\mathcal{D}_{a}^{\ast}(w,d)$ with respect to $w_{i}(t)$ for all $i\in a$, and by its strictly convexity.

From the functional calibration restriction~\eqref{equ1} and by the existence assumption on the inverse of the matrix $\sum_{i\in a}d_{i}q_{i}(t)\boldsymbol{X}_{i}(t)\boldsymbol{X}_{i}(t)^{\top}$ for all $t$, the Lagrange functional multiplier vector is determined by 
\[
\hat{\boldsymbol{\lambda}}(t)=\left(\sum_{i\in a}d_{i}q_{i}(t)\boldsymbol{X}_{i}(t)\boldsymbol{X}_{i}(t)^{\top}\right)^{-1}\left(\boldsymbol{\mu}_{X}(t)-\hat{\boldsymbol{\mu}}_{X}^{HT}(t)\right).
\]
Finally, replacing $\hat{\boldsymbol{\lambda}}(t)$ into the first order conditions, the calibration functional estimator $\hat{w}_{i}(t)$ of the Theorem is obtained.
\end{proof}

\begin{proof}[Proof of Theorem~\ref{theorem.2}]\vskip .1in 
\citet[Theorem 3, page 775]{Csiszar-84}.
\end{proof}

\begin{proof}[Proof of Lemma~\ref{lemma.1}] \vskip .1in
According to Theorem~\ref{theorem.2}, the maximum of the entropy $S(\nu\parallel\upsilon)$ over the class $\mathcal{V}=\left\lbrace \nu\ll\upsilon\colon Z_{\upsilon}(\boldsymbol{\lambda})<\infty\right\rbrace$ subject to the linear moment calibration constraint $\mathbb{E}_{\upsilon}\left[\mathcal{K}\mathcal{W}\right]=N\boldsymbol{\mu}_{X}(t)$ is attained at $d\nu^{\ast}/d\upsilon=Z_{\upsilon}^{-\text{1}}(\boldsymbol{\lambda}^{\ast})\exp\left\lbrace \langle\boldsymbol{\lambda}^{\ast},\mathcal{K}\mathcal{W}\rangle\right\rbrace$, where 
\begin{equation}\label{equ2}
\begin{split}Z_{\upsilon}(\boldsymbol{\lambda})
%& =\mathbb{E}_{\upsilon}\left[\exp\left\lbrace \langle\boldsymbol{\lambda},\mathcal{K}\mathcal{W}\rangle\right\rbrace \right]\\
 & =\exp\left\{ \mathbb{E}_{\upsilon}\left[\langle\boldsymbol{\lambda},\mathcal{K}\mathcal{W}\rangle\right]+\frac{1}{2}\mathbb{V}_{\upsilon}\left[\langle\boldsymbol{\lambda},\mathcal{K}\mathcal{W}\rangle\right]\right\} \\
 & =\exp\left\{ \sum_{i\in a}d_{i}\int_{0}^{1}\boldsymbol{\lambda}^{\top}(dt)\boldsymbol{X}_{i}(t)+\frac{1}{2}\int_{0}^{1}\left(\sum_{i\in a}\int_{0}^{1}K(s,t)\boldsymbol{\lambda}^{\top}(dt)\boldsymbol{X}_{i}(t)\right)^{2}ds\right\}, 
\end{split}
\end{equation}
owing to that $\mathbb{E}_{\upsilon}\left[d\mathcal{W}_{i}\left(s\right)\right]=0$, and $\mathbb{V}_{\upsilon}\left[d\mathcal{W}_{i}\left(s\right)\right]=ds$, $i\in a$.

Now we proceed with the problem of finding $\boldsymbol{\lambda}^{\ast}(dt)\in\mathcal{M}_{b}\left(C\left[0,1\right]^{q}\right)$, where $\mathcal{M}_{b}$ is the class of bounded continuous measures, such that minimizes 
\begin{equation}\label{equ3}
\begin{split}H_{\upsilon}(\boldsymbol{\lambda})
%& =\log Z_{\upsilon}(\boldsymbol{\lambda})-\langle\boldsymbol{\lambda},N\boldsymbol{\mu}_{X}\rangle\\
 & =\frac{1}{2}\int_{0}^{1}\left(\sum_{i\in a}\int_{0}^{1}K(s,t)\boldsymbol{\lambda}^{\top}(dt)\boldsymbol{X}_{i}(t)\right)\left(\sum_{i'\in a}\int_{0}^{1}K(s,t')\boldsymbol{\lambda}^{\top}(dt')\boldsymbol{X}_{i'}(t')\right)ds\\
 & \quad+\int_{0}^{1}\boldsymbol{\lambda}^{\top}(dt)\left(\sum_{i\in a}d_{i}\boldsymbol{X}_{i}(t)-N\boldsymbol{\mu}_{X}(t)\right)\\
 & =\frac{1}{2}\sum_{i\in a}\sum_{i'\in a}\int_{0}^{1}\int_{0}^{1}\int_{0}^{1}K(s,t)K(s,t')\boldsymbol{\lambda}^{\top}(dt)\boldsymbol{X}_{i}(t)\boldsymbol{X}_{i'}^{\top}(t')\boldsymbol{\lambda}(dt')ds\\
 & \quad+\int_{0}^{1}\boldsymbol{\lambda}^{\top}(dt)\left(\sum_{i\in a}d_{i}\boldsymbol{X}_{i}(t)-N\boldsymbol{\mu}_{X}(t)\right).
\end{split}
\end{equation}
The corresponding equation for $\boldsymbol{\lambda}^{\ast}(dt)$ that minimizes $H_{\upsilon}(\boldsymbol{\lambda})$ is given by 
\begin{equation}\label{equ4}
\sum_{i\in a}\sum_{i'\in a}\int_{0}^{1}\int_{0}^{1}K(s,t)K(s,t')\boldsymbol{X}_{i}(t)\boldsymbol{X}_{i'}^{\top}(t')\boldsymbol{\lambda^{\ast}}(dt')ds+\sum_{i\in a}d_{i}\boldsymbol{X}_{i}(t)=N\boldsymbol{\mu}_{X}(t),
\end{equation}
which can be rewritten as 
\[
\sum_{i\in a}\left[\int_{0}^{1}K(s,t)\left(\sum_{i'\in a}\int_{0}^{1}K(s,t')\boldsymbol{X}_{i'}^{\top}(t')\boldsymbol{\lambda}^{\ast}(dt')\right)ds+d_{i}\right]\boldsymbol{X}_{i}(t)=N\boldsymbol{\mu}_{X}(t),
\]
obtaining, by the moment calibration constraint~\eqref{restriction}, the Lemma's result. %that $\hat{w}_{i}(t)=\int_{0}^{1}K(s,t)\varpi^{\ast}(s)ds+d_{i}$ with $\varpi^{\ast}(s)=\sum_{i'\in a}\int_{0}^{1}K(s,t')X_{i'}^{\top}(t')\lambda^{\ast}(dt')$.
\end{proof}

\begin{proof}[Proof of Lemma~\ref{lemma.2}]\vskip .1in
For each $i\in a$, define a random variable $m_{i}\left(\left(a,b\right]\right)$ for $\left(a,b\right]\subset\left[0,1\right]$, 
\[
m_{i}\left(\left(a,b\right]\right)\triangleq\mathcal{W}_{i}(b)-\mathcal{W}_{i}(a)=\sum_{k=N(a)+1}^{N(b)}\xi_{ik}.
\]

By the Lévy-Khintchine formula for Lévy processes, the moment generating function of the $n$-dimensional compound Poisson process $\mathcal{\boldsymbol{\mathcal{W}}}(s)$ is given by 
\[
\mathbb{E}_{\upsilon}\left[\exp\left\{ \left\langle \boldsymbol{\alpha},\boldsymbol{\mathcal{W}}(s)\right\rangle \right\} \right]=\exp\left\{ s\gamma\int_{\mathbb{R}^{n}}\left(e^{\left\langle \boldsymbol{\alpha},\boldsymbol{\xi}_{k}\right\rangle }-1\right)u\left(d\boldsymbol{\xi}_{k}\right)\right\} ,\qquad\boldsymbol{\alpha}\in\mathbb{R}^{n},
\]
where $\boldsymbol{\xi}_{k}=\left(\xi_{1k},\ldots,\xi_{nk}\right)^{\top}$. This formula can be generalized for a continuous function $g(s)$ from $\left[0,1\right]$ to $\mathbb{R}$ and defining $\langle g(s),\mathcal{W}_{i}\rangle=\int_{0}^{1}g(s)d\mathcal{W}_{i}(s)$ for each $i\in a$, which is approximated by $\sum_{j=1}^{J}g\left(s_{j-1}\right)m_{i}\left(\left(s_{j-1},s_{j}\right]\right)$,
with $s_{j}=j/J$, $j=1,\ldots,J$. Thus, by the independence of $m_{i}\left(\left(a,b\right]\right)$, we have that 
\begin{equation}\label{equ5}
\begin{split}\mathbb{E}_{\upsilon}\left[\exp\left\{ \langle g(s),d\mathcal{W}_{i}\rangle\right\} \right] & =\lim_{J\rightarrow\infty}\prod_{j=1}^{J}\mathbb{E}_{\upsilon}\left[\exp\left\{ g\left(s_{j-1}\right)m_{i}\left(\left(s_{j-1},s_{j}\right]\right)\right\} \right]\\
 & =\lim_{J\rightarrow\infty}\prod_{j=1}^{J}\exp\left\{ \mathbb{E}_{\upsilon}\left[\exp\left\{ g\left(s_{j-1}\right)\xi_{i}\right\} \right]\right\} \\
 & =\lim_{J\rightarrow\infty}\prod_{j=1}^{J}\exp\left\{\frac{\gamma}{J}\int_{\mathbb{R}}\left(\exp\left\{ g\left(s_{j-1}\right)\xi_{i}\right\} -1\right)u\left(d\xi_{i}\right)\right\} \\
 & =\exp\left\{ \gamma\int_{0}^{1}ds\int_{\mathbb{R}}\left(\exp\left\{ g\left(s\right)\xi_{i}\right\} -1\right)u\left(d\xi_{i}\right)\right\},\quad\quad i\in a.
\end{split}
\end{equation}

Now, by the Theorem~\ref{theorem.2}, the maximum of the entropy $S$ over the class $\mathcal{V}$ subject to $\mathbb{E}_{\upsilon}\left[\mathcal{K}\mathcal{W}\right]=N\boldsymbol{\mu}_{X}(t)$ is achieved at $d\nu^{\ast}/d\upsilon=Z_{\upsilon}^{-1}(\boldsymbol{\lambda}^{\ast})\exp\left\lbrace \langle\boldsymbol{\lambda}^{\ast},\mathcal{KW}\rangle\right\rbrace$ with

\[
\begin{split}\langle\boldsymbol{\lambda},\mathcal{KW}\rangle & =\int_{0}^{1}\boldsymbol{\lambda}^{\top}(dt)\int_{0}^{1}\sum_{i\in a}K(s,t)\boldsymbol{X}_{i}(t)d\mathcal{W}_{i}(s)+\int_{0}^{1}\boldsymbol{\lambda}^{\top}(dt)\sum_{i\in a}d_{i}\boldsymbol{X}_{i}(t)\\
 & =\langle g(s),\mathcal{W}_{i}\rangle+\left\langle \boldsymbol{\lambda},\sum_{i\in a}d_{i}\boldsymbol{X}_{i}(t)\right\rangle,
\end{split}
\]
where $g(s)=\int_{0}^{1}\boldsymbol{\lambda}^{\top}(dt)\sum_{i\in a}K(s,t)\boldsymbol{X}_{i}(t)$.

\noindent Therefore,
\begin{equation}\label{equ6}
\begin{split}
Z_{\upsilon}(\boldsymbol{\lambda})
%& =\mathbb{E}_{\upsilon}\left[\exp\left\{ \langle g(s),d\mathcal{W}_{i}\rangle+\left\langle \boldsymbol{\lambda},\sum_{i\in a}d_{i}\boldsymbol{X}_{i}(t)\right\rangle \right\} \right]\\
% & =\mathbb{E}_{\upsilon}\left[\exp\left\{ \langle g(s),d\mathcal{W}_{i}\rangle\right\} \right]\exp\left\{ \left\langle \boldsymbol{\lambda},\sum_{i\in a}d_{i}\boldsymbol{X}_{i}(t)\right\rangle \right\} \\
 & =\exp\left\{ \gamma\int_{0}^{1}ds\int_{\mathbb{R}}\left(\exp\left\{ g\left(s\right)\xi_{i}\right\} -1\right)u\left(d\xi_{i}\right)\right\} \exp\left\{ \left\langle \boldsymbol{\lambda},\sum_{i\in a}d_{i}\boldsymbol{X}_{i}(t)\right\rangle \right\} \\
 & =\exp\left\{ \gamma\int_{0}^{1}ds\int_{\mathbb{R}}\left(\exp\left\{ g\left(s\right)\xi_{i}\right\} -1\right)u\left(d\xi_{i}\right)+\left\langle \boldsymbol{\lambda},\sum_{i\in a}d_{i}\boldsymbol{X}_{i}(t)\right\rangle \right\} 
\end{split}
\end{equation}

Finally, as in the proof of Lemma~\ref{lemma.1}, the problem is concentrated to find $\boldsymbol{\lambda}^{\ast}(t)$ such that minimizes 
\begin{equation}\label{equ7}
\begin{split}H_{\upsilon}(\boldsymbol{\lambda})
%& =\log Z_{\upsilon}(\boldsymbol{\lambda})-\langle\boldsymbol{\lambda},N\boldsymbol{\mu}_{X}\rangle\\
% & =\gamma\int_{0}^{1}ds\int_{\mathbb{R}}\left(\exp\left\{ g\left(s\right)\xi_{i}\right\} -1\right)u\left(d\xi_{i}\right)+\left\langle \boldsymbol{\lambda},\sum_{i\in a}d_{i}\boldsymbol{X}_{i}(t)-N\boldsymbol{\mu}_{X}(t)\right\rangle \\
 & =\gamma\int_{0}^{1}ds\int_{\mathbb{R}}\left(\exp\left\{ \int_{0}^{1}\boldsymbol{\lambda}^{\top}(dt)\sum_{i\in a}K(s,t)\xi_{i}\boldsymbol{X}_{i}(t)\right\} -1\right)u\left(d\xi_{i}\right)\\
 & \quad+\int_{0}^{1}\boldsymbol{\lambda}^{\top}(dt)\left(\sum_{i\in a}d_{i}\boldsymbol{X}_{i}(t)-N\boldsymbol{\mu}_{X}(t)\right).
\end{split}
\end{equation}

\noindent  The corresponding equation for $\boldsymbol{\lambda}^{\ast}(dt)$ that minimizes $H_{\upsilon}(\boldsymbol{\lambda})$ is given by
\begin{equation}\label{equ8}
\begin{split}
\sum_{i\in a}\left[\int_{0}^{1}K(s,t)\left(\int_{\mathbb{R}}\xi_{i}\exp\left\{ \sum_{i\in a}\int_{0}^{1}K(s,t)\xi_{i}\boldsymbol{X}_{i}^{\top}(t)\boldsymbol{\lambda}^{\ast}(dt)\right\} u\left(d\xi_{i}\right)\right)ds+d_{i}\right]\\
\times\boldsymbol{X}_{i}(t)=N\boldsymbol{\mu}_{X}(t),
\end{split}
\end{equation}
obtaining, by the moment calibration constraint~\eqref{restriction}, the Lemma's result. %that $\hat{w}_{i}(t)=\int_{0}^{1}K(s,t)\varpi^{\ast}(s)ds+d_{i}$ with $\varpi_{i}^{\ast}(s)=\int_{\mathbb{R}}\xi_{i}\exp\left\{ \sum_{i\in a}\int_{0}^{1}K(s,t)\xi_{i}X_{i}^{\top}(t)\lambda^{\ast}(dt)\right\} u\left(d\xi_{i}\right)$.
\end{proof}

%\begin{acknowledgements}
%If you'd like to thank anyone, place your comments here and remove the percent signs.
%\end{acknowledgements}

\bibliography{references_entropy}
\end{document}